\documentclass{fundam}

\usepackage[utf8]{inputenc}
\usepackage{amsmath}
\usepackage{amssymb}
\usepackage{latexsym}
\usepackage{graphicx}

\usepackage{enumitem}
\usepackage{cite}
\usepackage[misc]{ifsym}

\usepackage{algorithm}
\usepackage[noend]{algpseudocode}

\algnewcommand{\LeftComment}[1]{\Statex \(\triangleright\) #1}
\algrenewcommand{\algorithmiccomment}[1]{\hfill// #1}
\makeatletter

\newenvironment{breakablealgorithm}
  {
   \begin{center}
     \refstepcounter{algorithm}
     \hrule height.8pt depth0pt \kern2pt
     \renewcommand{\caption}[2][\relax]{
       {\raggedright\textbf{\fname@algorithm~\thealgorithm} ##2\par}%
       \ifx\relax##1\relax 
         \addcontentsline{loa}{algorithm}{\protect\numberline{\thealgorithm}##2}%
       \else 
         \addcontentsline{loa}{algorithm}{\protect\numberline{\thealgorithm}##1}%
       \fi
       \kern2pt\hrule\kern2pt
     }
  }{
     \kern2pt\hrule\relax
   \end{center}
  }
\makeatother

\makeatletter
\renewcommand*\env@matrix[1][*\c@MaxMatrixCols c]{%
\hskip -\arraycolsep
\let\@ifnextchar\new@ifnextchar
\array{#1}}
\makeatother

\begin{document}

\setcounter{page}{91}
\publyear{22}
\papernumber{2154}
\volume{189}
\issue{2}

  \finalVersionForARXIV

\title{Error Correction for Discrete Tomography}

\author{Matthew Ceko\thanks{Address  for correspondence: School of Physics
                                         and Astronomy Monash University, Melbourne, Australia}
\\
School of Physics and Astronomy Monash University, Melbourne, Australia\\
matthew.ceko@gmail.com
\and
Lajos Hajdu\\
Institute of Mathematics, University of Debrecen,  Debrecen,  Hungary\\
hajdul@science.unideb.hu
\and
Rob Tijdeman\\
Mathematical Institute, Leiden University, Leiden, The Netherlands\\
tijdeman@ziggo.nl
}

\runninghead{M. Ceko and all.}{Error Correction for Discrete Tomography}

\maketitle

\vspace{-10mm}
\hfill{\footnotesize\sl Dedicated to the memory of Carla Peri.}

\vspace*{-4mm}
\begin{abstract}
Discrete tomography focuses on the reconstruction of functions from their line sums in a finite number $d$ of directions. In this paper we consider functions $f: A \to R$ where $A$ is a finite subset of $\mathbb{Z}^2$ and $R$ an integral domain.
Several reconstruction methods have been introduced in the literature. Recently Ceko, Pagani and Tijdeman developed a fast method to reconstruct a function with the same line sums as $f$. Up to here we assumed that the line sums are exact. Some authors have developed methods to recover the function $f$ under suitable conditions by using the redundancy of data. In this paper we investigate the case where a small number of line sums are incorrect as may happen when discrete tomography is applied for data storage or transmission. We show how less than $d/2$ errors can be corrected and that this bound is the best possible. Moreover, we prove that if it is known that the line sums in $k$ given directions are correct, then the line sums in every other direction can be corrected provided that the number of wrong line sums in that direction is less than $k/2$.
\end{abstract}


\section{Introduction}
\label{intro}

We consider functions $f: A \to R$ where $A= \{(i,j) \in \mathbb{Z}^2: 0 \leq i <m, 0 \leq j <n\}$ for given positive integers $m,n$ and integral domain $R$, e.g. $\mathbb{Z}, \mathbb{R}$ or a finite field. We assume that $f$ is unknown, but that its line sums in a positive number $d$ of directions are given. The line sums are often referred to as X-rays to highlight the link between discrete tomography and computed tomography scans.
This type of discrete tomography problem has been widely studied, see  e.g. \cite{ag19, dart, cpst, dgci2016, gg, tenyears, ht01, her, hk07, sb}. Discrete tomography originated in the study of crystals which may be damaged when many X-rays are used \cite{fishepp, sf}. In such applications the possible values of the line sums are linear combinations of a finite set of positive real numbers. Later applications of discrete tomography were developed where the co-domain of $f$ can be chosen, such as distributed storage \cite{bgh, png}, watermarking \cite{ac, uga}, image compression \cite{ka08} and erasure coding \cite{csgkn, nspk, pinp}.

It makes an essential difference whether the line sums are exact or not. If they are exact then there is at least one function satisfying the line sums, but there may be infinitely many.
In 1978 Katz \cite{ka77} gave a necessary and sufficient condition for the uniqueness of the solution. The structure of possible solutions of a discrete tomography problem has been studied by numerous authors. It turns out that proving the existence of a solution, and in case of existence a subsequent reconstruction, can be very hard if the range of the function on $A$ is restricted to a fixed finite set. In 1991 Fishburn, Lagarias, Reeds and Shepp \cite{fishepp} gave necessary and sufficient conditions for uniqueness of reconstruction of functions $f: A \to \{1,2, \dots, N\}$ for some positive integer $N>1$. See also the thesis of Wiegelmann \cite{wie}. In 1999 and 2000 Gardner, Gritzmann and Prangenberg \cite{ggp99, ggp00} showed under very general conditions that proving the existence or uniqueness of a function $f: A \to \mathbb{N}$ from its line sums in $d$ directions is NP-complete. The crux of the NP-results is that the co-domain is not closed under subtraction.

If the co-domain $R$ is an integral domain, then linear algebra techniques such as Gauss elimination provide polynomial time algorithms. This is useful for the practical reconstruction of $f$, see e.g. Batenburg and Sijbers \cite{dart}. In the present paper we investigate the theoretical structure of solutions and leave such computational techniques aside. If the solution is not unique, then any two solutions differ by a so-called ghost, a nontrivial  function $g: A \to \mathbb{R}$ for which all the line sums in the $d$ directions vanish. In 2001 Hajdu and Tijdeman \cite{ht01} gave an explicit algebraic expression for the ghost of minimal size and showed that every ghost is a linear combination of shifts of it. It implies that arbitrary function values can be given to a certain set of points of $A$ and that thereafter the function values of the other points of $A$ are uniquely determined by the line sums, see Dulio and Pagani \cite{dupa}.

Suppose $A$ is an $m$ by $n$ grid and the line sums of a function $f: A \to R$ in $d$ directions are known. Recently a method was developed to construct a function $g: A \to R$ which has the same line sums as $f$ has in time linear in $dmn$ as to the number of operations such as addition and multiplication. This development started with four papers of Dulio, Frosini and Pagani \cite{urrpdt, deda, dgci2016, 3dirext} with fast reconstruction results for corner regions of $A$ in case $d=2$ or $3$. Subsequently Pagani and Tijdeman \cite{pt19} did so for general $d$. In particular, their approach enables one to reconstruct $f$, if $f$ is the only function which satisfies the line sums in the $d$ directions. Finally Ceko, Pagani and Tijdeman \cite{cst} developed an algorithm to construct a function $g: A \to R$ in time linear in $dmn$ such that $g$ has the same line sums as $f$. This yields a parameter representation of all the functions $g: A \to R$ which have the same line sums as $f$.
We think it is unlikely that there exists a general reconstruction method which requires essentially less than $\mathcal{O}(dmn)$ operations, if the solution is unique.

A remaining problem is to construct the most likely consistent set of line sums if the measured line sums contain errors. The most common cause of inconsistency of line sums is noise. This happens, for example, if the line sums are approximations of line integrals. Here the assumption is that many line sums may not be exact, but that for each line sum the difference between measured and actual sum is small. Many algorithms have been developed to deal with this situation in case $R = \mathbb{R}$, often approximation methods which work well in practice but do not guarantee optimality. See for example Parts 2 of the books edited by Herman and Kuba \cite{hk99, hk07} and the paper by Batenburg and \linebreak Sijbers \cite{dart}.

The consistent set of line sums nearest to the measured line sums in case of $R=\mathbb R$ can be constructed as follows. Consider the incidence matrix $B$ of the lines in the $d$ given directions and intersecting $A$. This is an $L$ by $mn$ matrix where $L$ is the total number of such lines. Then the range of $B$ forms a subspace in $\mathbb R^L$. By constructing the orthogonal projection of the vector of the measured line sums onto this subspace, we obtain the consistent line sum vector $b_0$ nearest to the measured one. Theorem 5.5.1 of \cite{gl} provides the standard tool from linear algebra to compute the vector $b_0$. For more details on this procedure see Theorem 4.1 of \cite{ht13}.

In the present paper we deal with another type of errors, viz. errors which may be arbitrary large, but are small in number. For literature in this direction, see e.g. \cite{bgh, csgkn, nspk, png}. Again the properties of the co-domain of $f$ make an essential difference. Alpers and Gritzmann \cite{ag06} showed that for functions $f: A \to \{0,1\}$ the Hamming distance between any two solutions with equal cardinality of the lattice sets is $2(d-1)$. They remarked that the problem of determining how the individual measurements should be corrected in order to provide consistency of the data is NP-complete whenever $d \geq 3$. The situation is again totally different and much more favourable if the co-domain of $f$ is an integral domain.

A common approach for error correction in linear systems involves solving an $L_1$ minimization problem \cite{don, wri}. Candes, Romberg and Tao showed that an object may be recovered exactly from incomplete frequency samples via convex optimization \cite{can}. Chandra, Svalbe, Gu\'edon, Kingston and Normand \cite{csgkn} used redundant image regions to reconstruct the original function $f$ in linear time. They chose $m,n$ and the directions appropriately and assumed that all the line sums in some directions were wrong. This result is comparable with our Theorem \ref{onedir} which, however, is valid for any $m,n$ and any finite set of directions. In Theorem \ref{mainth} we show that if $f: A \to R$ and $g: A \to R$ do not have the same line sums, then at least $d$ line sums are different and in case of $<d/2$ wrong line sums reconstruction of $f$ can be done in polynomial time. (Even $<d$ wrong line sums if it is known which line sums are wrong). A simple example shows that here the bounds $d$ and $d/2$, respectively, are the best possible (see Section \ref{sec3}). After having corrected the wrong line sums, we have a consistent set of line sums and the method from \cite{cst, pt19} can be applied to find the optimal solution in the above described sense.

Theorems \ref{mainth} and \ref{onedir} are stated in the next section. Moreover, in Sections 9-11 we give a pseudocode algorithm and an example, and prove that the complexity is $\mathcal{O}(d^4mn)$ operations. Of course, it may happen that the number of errors is much smaller than $d/2$. We introduce $F$ for the total number of errors  and $G$ for the maximum of the number of wrong line sums in some direction as parameters so that the amount of computation can be reduced if it is expected that there are far less than $d/2$ errors. Here $F$ and $G$ can be freely chosen such that $G \leq F < d/2$.

In the proofs of the theorems we use the fact that there is redundancy in the information given by the line sums. Hajdu and Tijdeman \cite{ht17} pursued an analysis of the redundacy by Van Dalen \cite{bvd}. Their line sum relation lemma (Lemma \ref{relat}) is the basis of the present paper. Besides, some properties of Vandermonde determinants are derived and used. By the redundacy of data the values of the wrong line sums do not matter. The lines with correct line sums are detected and the right values of the wrong line sums are derived from the correct line sums.

\section{The main results}
\label{sec2}

Let $d,m,n$ be positive integers and $A= \{(i,j) \in \mathbb{Z}^2 : 0 \leq i <m, 0 \leq j <n \}$. Let $D = \{(a_h,b_h)  : h=1,2, \dots, d \}$ be a set of pairs of coprime integers with $a_h \geq 0$ and $b_h=1$ if $a_h=0$.
 We call the elements of $D$ directions. For $f : A \to \mathbb{R}$ we define the line sum in the direction of $(a_h,b_h)$ by
\begin{equation} \label{defm}
    \ell_{h,t} = \sum_{(i,j) \in A,~ b_hi-a_hj=t} f(i,j)
\end{equation}
for $h=1$, $2$, $\dots$, $d$ and $t \in \mathbb{Z}$.
Denote for all $h$ and $t$ by $\ell_{h,t}^*$ the corresponding measured line sum. We call line sums with $\ell_{h,t}^* = \ell_{h,t}$ correct line sums and others wrong line sums.
In this paper we suppose that all the line sums in the directions of $D$ are measured and that there are less than $d/2$ wrong line sums and show how to correct them.

\begin{theorem}
\label{mainth}
Let $d,m,n$ be positive integers and let $A$ and $D$ be as defined above.
Let $f : A \to \mathbb{R}$ be an unknown function such that for $h=1$, $2$, $\dots$, $d$ the line sums $\ell^*_{h,t}$ in the direction of $(a_h,b_h)$ are measured with in total less than $d/2$ wrong line sums. Then the correct line sums can be determined.
\end{theorem}

It is remarkable that the bound depends only on $d$ and is independent of $m, n$ and the directions themselves. The restriction on the entries of the directions serves to choose one of the two directions $(a,b)$ and $(-a,-b)$ which provide the same line sums.

If $k$ directions with only correct line sums are known and there are not too many wrong line sums in some other direction, then these wrong line sums can be corrected:

\begin{theorem} \label {onedir}
Suppose, in the notation of Theorem \ref{mainth}, all the line sums in directions $(a_h,b_h)$ are known to be correct for $h=1, 2, \dots, k$. Then the line sums can be corrected in each direction with less than $k/2$ wrong line sums, and, moreover, in each direction with at most $k$ wrong line sums where it is known which line sums are wrong.
\end{theorem}

 In Section \ref{sec3} we show that the bound $d/2$ cannot be improved. Sections \ref{sec4} and \ref{sec5} contain results related to Vandermonde determinants. The line sum relation lemma is proved in Section \ref{sec6}. Theorems \ref{onedir} and \ref{mainth} are derived in Sections \ref{oned} and \ref{sec7}, respectively. In Section \ref{sec8} a pseudocode is provided, which details the steps of the algorithm to find the correct line sums. An example in Section \ref{sec9} illustrates the algorithm. Section \ref{sec10} provides an analysis of the complexity of the algorithm. In the final section we state some conclusions.

\section{An example which shows that Theorem \ref{mainth} cannot be improved}
\label{sec3}

Change the value of $f$ at one element of $A$. Then exactly $d$ line sums change. Thus $d$ of the pairs of corresponding line sums are different. It follows that the bound $d/2$ in Theorem \ref{mainth} is the best possible.
\vskip.2cm
\noindent {\bf Example 1.} Consider the function $f: A \to \mathbb{Z}$ with one unknown value indicated by ?.\vspace*{1mm}
\[
\begin{array}{|c|c|c|c|}
\hline
2&6&5&4 \\
\hline
3&?&2&0 \\
\hline
5&1&4&2 \\
\hline
6&3&1&4 \\
\hline
\end{array}\vspace*{1mm}
\]
Suppose the measured horizontal line sum through ? is 7, the measured vertical line sum through ? is 12 and both measured diagonal line sums through ? are 13. Then $d=4$, the horizontal and vertical line sums suggest that the value of ? is 2 whereas the diagonal line sums indicate that it should be 3. Both values result in $2 = d/2$ wrong line sums. \QED

\medskip
Obviously this can be generalized for arbitrarily large $m,n$ and $d$.

\section{Vandermonde equations with variable coefficients}
\label{sec4}

Let $r$ be a positive integer. Let $c_0$, $c_1$, $\dots$, $c_{2r-1}$ and $t_1$, $t_2$, $\dots$, $t_r$ be given real numbers with $t_1$, $t_2$, $\dots$, $t_r$ distinct. It is well known and very useful that a system of linear equations
\begin{equation}
\label{eqvarm}
    \sum_{i=1}^r  t_i^j x_i = c_j
\end{equation}
for $j=0$, $1$, $\dots$, $r-1$ in unknowns $x_1$, $x_2$, $\dots$, $x_r$ has a unique solution which can be found by using a Vandermonde matrix. In this section we show how to solve the system of equations \eqref{eqvarm}
for $j=0$, $1$, $\dots$, $2r-1$, if both  $x_1$, $x_2$, $\dots$, $x_r$ and $t_1$, $t_2$, $\dots$, $t_r$ are unknowns.

\medskip
The method is based on the following lemmas.

\begin{lemma}
\label{vdmsq}
Let M be the $r$ by $r$ matrix with entries $M_{i,j} = \sum_{h=1}^r t_h^{i+j}$ for $i$, $j=0$, $1$, $\dots$, $r-1$.
Then
\begin{equation*}
    \det(M) = \prod_{1 \leq i < j \leq r} (t_j-t_i)^2.
\end{equation*}
\end{lemma}

\begin{proof}
Observe that $M=V^T \cdot V$ where $V$ is the Vandermonde matrix with $V_{i,j} = t_i^{j}$ for $i=1$, $2$, $\dots$, $r$ and $j=0$, $1$, $\dots$, $r-1$. Therefore
\begin{equation*}
    \det(M) =  (\det(V))^2 = \left( \prod_{i<j} (t_j-t_i) \right)^2.
\end{equation*}

\vspace*{-9mm}
\end{proof}

\begin{lemma}
\label{reccB}
Let $c_0$, $c_1$, $\dots$, $c_{2r-1}$ be given real numbers. If $t_1$, $t_2$, $\dots$, $t_r$ and $x_1$, $x_2$, $\dots$, $x_r$ satisfy \eqref{eqvarm} for $j = 0,1, \dots 2r-1$, then
\begin{equation}
\label{relcB}
    c_{j+r} -c_{j+r-1}B_1 + \dots + (-1)^rc_jB_r = 0
\end{equation}
for $j=0$, $1$, $\dots$, $r-1$ where $B_1$, $B_2$, $\dots$, $B_r$ are defined by
\begin{equation}
\label{defB}
    (z-t_1)(z-t_2) \dots (z-t_r)= z^r -B_1z^{r-1} + \dots + (-1)^r B_r.
\end{equation}
\end{lemma}

\begin{proof}
We have\vspace*{-2.5mm}
\begin{align*}
    c_{j+r} = \sum_{i=1}^r  t_i^{j+r} x_i &= \sum_{i=1}^r x_i \sum_{h=1}^r (-1)^{h-1}B_ht_i^{j+r-h} \\
    &= \sum_{h=1}^r (-1)^{h-1} B_h \sum_{i=1}^r t_i^{j+r-h} x_i \\
    &= \sum_{h=1}^r (-1)^{h-1} B_h c_{j+r-h}.
\end{align*}

\vspace*{-9mm}
\end{proof}

We are now ready to show how system \eqref{eqvarm} can be solved if both  $x_1$, $x_2$, $\dots$, $x_r$ and $t_1$, $t_2$, $\dots$, $t_r$ are unknowns.

\begin{lemma}
\label{thmdir}
Let $r$ be a positive integer. Let $c_0$, $c_1$, $\dots$, $c_{2r-1}$ be given real numbers. If nonzero real numbers $x_1$, $x_2$, $\dots$, $x_r$ and distinct real numbers $t_1$, $t_2$, $\dots$, $t_r$ satisfy \eqref{eqvarm} for $j = 0,1, \dots, 2r-1$, then $B_1$, $B_2$, $\dots$, $B_r$ defined by \eqref{defB} can be determined by solving the linear system \eqref{relcB} for $j=0$, $1$, $\dots$, $r-1$. Subsequently $t_1$, $t_2$, $\dots$, $t_r$  can be found  by computing the zeros of the polynomial
\begin{equation}
\label{compm}
    z^r - B_1 z^{r-1} + B_2z^{r-2} + \dots + (-1)^r B_r.
\end{equation}
If $t_1$, $t_2$, $\dots$, $t_r$ are chosen, the values of $x_1$, $x_2$, $\dots$, $x_r$ can be found by solving system \eqref{eqvarm} for $j=0$, $1$, $\dots$, $r-1$.
\end{lemma}

\begin{proof}
First we apply Lemma \ref{reccB}, where, by \eqref{eqvarm}, we have to solve a system of $r$ linear equations in $r$ unknowns $B_1$, $B_2$, $\dots$, $B_r$  with coefficient matrix $M^*$ with $M^*_{i,j}= (-1)^{r-i} \sum_{h=1}^r t_h^{i+j}x_h$ for $i$, $j=0$, $1$, $\dots$, $r-1$. Note that $$\det(M^*) = \pm x_1x_2 \cdots x_r \cdot \det(M)$$ where $M$ is the matrix from Lemma \ref{vdmsq}. Since $\det(M^*)$ is nonzero by Lemma \ref{vdmsq}, we can solve the system of $r$ linear equations and so determine the numbers $B_1$, $B_2$, $\dots$, $B_r$. By computing the zeros of \eqref{compm} the numbers $t_i$ are found. Note that the numbers $t_1$, $t_2$, $\dots$, $t_r$ cannot be distinguished and we may assume $t_1 < t_2 < \dots < t_r$. The expression for $x_i$ follows from solving system \eqref{eqvarm} for $j=0$, $1$, $\dots$, $r-1$ using Cramer's rule.
\end{proof}

\noindent {\bf Remark.} Later on we apply Lemma \ref{thmdir} in such a way that the number $t_i$ corresponds to the line $b_hx-a_hy=t_i$ for which the line sum is wrong and the corresponding number $x_i$ is equal to the difference between the measured line sum and the correct line sum. We refer to Section \ref{sec8}, in particular formula \eqref{relVW}, for the way the $x_i$'s are computed in practice.
\vskip.1cm
We conclude this section with a simple application of the Vandermonde determinant.

\begin{lemma}
\label{vdmo}
Let $t_1$, $t_2$, $\dots$, $t_r$ be distinct integers. If $ \sum_{i=1}^r t_i^jx_i = 0$ for $j=0$, $1$, $\dots$, $k-1$, then $k<r$ or $x_1=x_2= \dots = x_r=0.$
\end{lemma}

\begin{proof}
It follows from $ \sum_{i=1}^r t_i^jx_i = 0$ for $j=0$, $1$, $\dots$, $r-1$ that the $t_i$'s are not distinct or all $x_i$'s are 0.
\end{proof}

\section{A Vandermonde-related determinant}
\label{sec5}

We prove the following result.
\begin{lemma}
\label{determ}
Let $a_1$, $a_2$, $\dots$, $a_{2k}$, $b_1$, $b_2$, $\dots$, $b_{2k}$ be variables. Set $W_{u,v} = a_ub_v-a_vb_u$ for $u, v = 1, 2, \dots, 2k$. Let $M$ be a $k \times k$ matrix with entries $M_{h,H} = \left(\prod_{i=1}^k W_{i,k+H} \right) / W_{h,k+H}$ for $h,H=1, \dots, k$. Then\vspace*{-1mm}
\begin{equation*}
    \det(M) = (-1)^{k(k-1)/2} \left( \prod_{1\leq h_1 <h_2 \leq k} W_{h_1,h_2} \right) \left(\prod_{ 1 \leq H_1  <H_2 \leq k} W_{k+H_1,k+H_2} \right).
\end{equation*}
\end{lemma}

\begin{proof} Clearly, we can consider det$(M)$ to be a polynomial in the unique factorization domain ${\mathbb R}[a_1,a_2,\dots$, $a_{2k},b_1,b_2,\dots,b_{2k}]$. The degree of det$(M)$ equals $2 k^2-2k$. For each $h_1,h_2$ with $1 \leq h_1 < h_2 \leq k$ the columns numbered $h_1$ and $h_2$ are proportional if $a_{h_1}b_{h_2} = a_{h_2}b_{h_1}$ which implies det$(M)$=0. Therefore det$(M)$ is divisible by $W_{h_1,h_2}$. Similarly, for each $H_1,H_2$ with $1 \leq H_1 < H_2 \leq k$ the rows numbered $H_1$ and $H_2$ are proportional if $a_{k+H_1}b_{k+H_2} = a_{k+H_2}b_{k+H_1}$ which implies that det$(M)$ is divisible by $W_{k+H_1,k+H_2}$.
The product of these distinct and coprime factors,
\begin{equation*}
    \left( \prod_{1\leq h_1 < h_2 \leq k} W_{h_1,h_2} \right) \left(\prod_{ 1 \leq H_1  <H_2\leq k} W_{k+H_1, k+H_2} \right),
\end{equation*}
has degree $2 k^2-2k$ too. Therefore there is a real number $c$ such that
\begin{equation*}
    \det(M) = c  \left( \prod_{1\leq h_1 <h_2 \leq k} W_{h_1,h_2} \right) \left(\prod_{ 1 \leq H_1  <H_2 \leq k} W_{k+H_1,k+H_2} \right).
\end{equation*}
Since $a_ub_v$ is lexicographically smaller than $a_vb_u$ for $u<v$, we infer that
\begin{equation*}
c \cdot a_1^{k-1}a_2^{k-2} \cdots a_{k-1} \cdot a_{k+1}^{k-1}a_{k+2}^{k-2} \cdots a_{2k-1} \cdot b_{2} b_{3}^2 \cdots b_{k}^{k-1} \cdot b_{k+2} b_{k+3}^2 \cdots b_{2k}^{k-1}
\end{equation*}
is the smallest lexicographic element in the expansion of $\det(M)$.
We claim that on comparing the exponents it turns out that this term can only be obtained by developing the main diagonal of $M$. The first column is the only one containing $a_{k+1}$'s.
Since no $b_1$ should be chosen, the only possibility is to choose $-a_{k+1}b_2, -a_{k+1}b_3, \dots, -a_{k+1}b_k$ from the leftmost element of the first column.
The second column is the only one containing $a_{k+2}$'s. Therefore it has to be chosen $k-2$ times and the other has to involve $a_1$. Since $b_2$ should not be chosen anymore, we choose $a_1b_{k+2}$ and $-a_{k+2}b_3, -a_{k+2}b_4, \dots, -a_{k+2}b_k$ from the element at the second column of the main diagonal. Continuing in this way it turns out that the only possible choice of the factors is in the expansion of entry $M_{h,h}$ the term with factors
\begin{equation*}
a_1b_{k+h}, a_2b_{k+h}, \dots, a_{h-1}b_{k+h}, -a_{k+h}b_{h+1}, -a_{k+h}b_{h+2}, \dots, -a_{k+h}b_{k},
\end{equation*}
for $h=1,2, \dots, k$.
Since the coefficient of the resulting product is $(-1)^{k(k-1)/2}$, we conclude that $c = (-1)^{k(k-1)/2}$.
\end{proof}

\noindent {\bf Example 2.} For $k=3$ the matrix $M$ is as follows, where the chosen elements to obtain the smallest lexicographic element are boldface.
\[ \left(
\begin{array} {ccc}
(a_2b_4-{\bf a_4b_2})(a_3b_4-{\bf a_4b_3}) & (a_2b_5-a_5b_2)(a_3b_5-{a_5b_3}) & (a_2b_6-{a_6b_2})(a_3b_6-{a_6b_3}) \\
(a_1b_4- a_4b_1)(a_3b_4-{a_4b_3}) & ({\bf a_1b_5}-{a_5b_1})(a_3b_5-{\bf a_5b_3}) & (a_1b_6-{a_6b_1})(a_3b_6-{a_6b_3}) \\
(a_1b_4-{a_4b_1})(a_2b_4-{a_4b_2}) & (a_1b_5-{a_5b_1})(a_2b_5-{a_5b_2}) & ({\bf a_1b_6}-{a_6b_1)(\bf a_2b_6}-{a_6b_2})
\end{array}
\right) \]

\vskip.2cm
An alternative version of Lemma \ref{determ} reads as follows.
\begin{corollary}
Let $a_1$, $a_2$, $\dots$, $a_{2k}$, $b_{1}$, $b_{2}$, $\dots$, $b_{2k}$ be reals. Set $W_{u,v} = a_ub_v-a_vb_u$ for $u, v = 1, 2, \dots, 2k$. Let $M^* = \{M^*_{h,H} : h = 1,2, \dots, k; H=1, 2, \dots, k\}$ be the matrix with entries $M^*_{h,H} =1 /W_{h,k+H}$. Then
\begin{equation*}
    \det(M^*) = (-1)^{k(k-1)/2} \left( \prod_{1\leq h_1 <h_2 \leq k} W_{h_1,h_2} \right) \left(\prod_{ 1 \leq H_1  < H_2 \leq k} W_{k+H_1,k+H_2} \right) \left(\prod_{h=1}^k \prod_{H=1}^{k} W_{h,k+H} \right)^{-1}.
\end{equation*}
\end{corollary}

\begin{proof}
Note that $M_{h,H} = (\prod_{i=1}^k W_{i,k+H}) \cdot M^*_{h,H}$ and that the factor within brackets is independent of $h$. Hence,
$$\det(M) = \left(\prod_{H=1}^k \prod_{i=1}^k W_{i,k+H}\right) \det(M^*).$$
\end{proof}

\section{The line sum relation lemma}
\label{sec6}

The following result is of fundamental importance in  our present study. It follows from Lemma 4.1 of \cite{ht17}.
For the convenience of the reader we give a direct proof here.

\begin{lemma} \label{relat}
Let $A, D, f, \ell_{h,t}$ be as in Section \ref{sec2}. Let $K$ be a subset of $\{1,2, \dots, d\}$. For $h = 1$, $2$, $\dots$, $k := |K| \geq 2$ define $E_{h,K}$ by
\begin{equation}
\label{defE}
E_{h,K} = (-1)^{h-1} \prod_{i,j \in K,~ i <j,~ i,j \not= h} (a_ib_j - a_jb_i).
\end{equation}
Then $$\sum_{h \in K}  E_{h,K} \sum_{t \in \mathbb{Z} }t^{k-2} \ell_{h,t} =0.$$
\end{lemma}

\begin{proof}
 Without loss of generality we may assume that $K = \{1, 2, \dots, k\}$ with $D_h = (a_h,b_h)$ for $h=1,2, \dots, k$.
Put ${\bf a}^s = (a_1^s, a_2^s, \dots, a_k^s)$ and ${\bf b}^s = (b_1^s, b_2^s, \dots, b_k^s)$ for $s = 0,1,2, \dots$.
We denote the determinant of the $m \times m$ matrix with $i$-th column vector ${\bf x}_i = (x_{1,i}, \dots, x_{m,i})$ by det$({\bf x}_1, \dots, {\bf x}_m)$.
Furthermore, we denote the determinant of the matrix which we obtain by omitting its first column vector and its $h$-th row vector by det$({\bf x}_2, \dots, {\bf x}_m)_h$.

\medskip
Obviously, for $s=0,1, \dots, k-2$ we have
$$ {\rm det}({\bf a}^s{\bf b}^{k-2-s}, {\bf a}^{k-2}, {\bf a}^{k-3}{\bf b}, {\bf a}^{k-4}{\bf b}^2, \dots, {\bf b}^{k-2}) = 0.$$
(Here the product of vectors is defined termwise.) By developing by the first column we obtain, for $s = 0,1, \dots, k-2,$
$$ \sum_{h=1}^k (-1)^{h-1}a_h^sb_h^{k-2-s} {\rm det}({\bf a}^{k-2}, {\bf a}^{k-3} {\bf b}, {\bf a}^{k-4}{\bf b}^2, \dots, {\bf b}^{k-2})_h = 0.$$
Observe that $(-1)^{h-1}{\rm det}({\bf a}^{k-2}, {\bf a}^{k-3} {\bf b}, {\bf a}^{k-4}{\bf b}^2, \dots, {\bf b}^{k-2})_h$ is the Vandermonde determinant $E_{h,k}$.
Hence, for arbitrary $(i,j) \in A$,
$$ \sum_{h=1}^k  (b_hi-a_hj)^{k-2} E_{h,K} = \sum_{s=0}^{k-2} {k-2 \choose s} i^{k-s-2}j^s \sum_{h=1}^ka_h^sb_h^{k-2-s} E_{h,k} = 0.$$
Since for every direction any element of $A$ is on exactly one line in that direction, we get
$$0 = \sum_{(i,j) \in A} f(i,j) \sum_{h=1}^k  (b_hi-a_hj)^{k-2} E_{h,K} = \sum_{h=1}^k \sum_{t \in \mathbb{Z}} ~\sum_{(i,j) \in A, b_hi-a_hj=t}~ f(i,j) (b_hi-a_hj)^{k-2}E_{h,K}.$$
Thus
$$0 = \sum_{h=1}^k \sum_{t \in \mathbb{Z}} t^{k-2} E_{h,K} ~\sum_{(i,j) \in A, b_hi-a_hj=t}~ f(i,j) = \sum_{h=1}^k \sum_{t \in \mathbb{Z}}E_{h,K}t^{k-2} \ell_{h,t}.$$

\vspace*{-9mm}
\end{proof}

\section{Error correction of line sums in one direction} \label{oned}
For the proof of Theorem 2.2 we combine the preceding lemmas.

\medskip\smallskip
\noindent  \textbf{Proof of Theorem 2.2}\\
Number the directions such that $D_1, D_2, \dots, D_k$ are directions with only correct line sums and direction $D_H$ for some fixed $H$ with $k < H \leq d$ may have wrong line sums.
For $j \leq k$ we apply Lemma \ref{relat} to the set $K_{j,H}:=\{1,2,\dots, j, H\}$,
\begin{equation}
\label{casek}
    \sum_{h=1}^{j}  E_{h,K_{j,H}} \sum_{t \in \mathbb{Z}} t^{j-1}\ell_{h,t} + E_{H,K_{j,H}} \sum_{t \in \mathbb{Z}} t^{j-1}\ell_{H,t} =0.
\end{equation}
Define and compute
\begin{equation}
\label{case*}
    c^*_{j,H} =   \sum_{h=1}^{j}  E_{h,K_{j,H}} \sum_{t \in \mathbb{Z}} t^{j-1}\ell^*_{h,t} + E_{H,K_{j,H}} \sum_{t \in \mathbb{Z}} t^{j-1}\ell^*_{H,t}.
\end{equation}
From \eqref{casek} and  \eqref{case*} we obtain,
\begin{equation}
\label{diffk}
    \sum_{h=1}^{j}  E_{h,K_{j,H}} \sum_{t \in \mathbb{Z}} t^{j-1}(\ell^*_{h,t}-\ell_{h,t}) + E_{H,K_{j,H}} \sum_{t \in \mathbb{Z}} t^{j-1}(\ell^*_{H,t}-\ell_{H,t})= c^*_{j,H}.
\end{equation}
The choice of $D_1$, $D_2$, $\dots$, $D_k$ implies $\ell^*_{h,t} = \ell_{h,t}$ for $h=1$, $2$, $\dots$, $k$ and all $t \in \mathbb{Z}$.  Thus
\begin{equation}
\label{dirH}
E_{H,K_{j,H}} \sum_{t \in \mathbb{Z}} t^{j-1}(\ell^*_{H,t} - \ell_{H,t})= c^*_{j,H}
\end{equation}
for $j=1$, $2$, $\dots$, $k$. Notice that by our assumption there are at most $k$ non-zero terms $\ell^*_{H,t} - \ell_{H,t}$, for $t= t_1, t_2, \dots, t_{r}$, say.
Then we have a system of linear equations \eqref{eqvarm} with $x_i = \ell^*_{H,t_i} - \ell_{H,t_i}, c_j = c^*_{j,H}/ E_{H,K_{j,H}}$. If $t_1, t_2, \dots, t_r$ are known, then we can simply solve system \eqref{eqvarm} and find $\ell^*_{H,t_i} - \ell_{H,t_i}$ for $i=1,2, \dots, r$ and determine the correct line sum $\ell_{H,t_i}$ for the line $b_Hx - a_Hy = t_i$ for $i=1,2, \dots, r$. If it is unknown which lines $b_Hx - a_Hy = t_i$ have wrong line sums, then $r < k/2$. Let $I$ be the largest integer less than $k/2$. Then we consider the system of linear equations
$$ \sum_{i=1}^I t_i^{j+h} x_i = c_{j+h} ~~~~{\rm for}~~j=0,1, \dots, I; h = 0,1,\dots, I.$$
The rank $r$ of the matrix with element $c_{j+h}$ for $j,h = 1, 2, \dots, I$ equals the number $r$ of line sums $\ell_{H,t_i}$ with wrong line sums.
Lemma \ref{thmdir} enables us to compute successively $B_1, B_2, \dots, B_r$ and $t_1, t_2, \dots, t_r$, indicating the lines $b_Hx-a_Hy=t_i$ where the wrong line sums are, and $x_1, x_2, \dots, x_r$, which represent the errors $\ell^*_{H,t_i} - \ell_{H,t_i}$ for $i=1,2, \dots, r$. Thus we can compute the correct line sums $\ell_{H,t_i}$ for $i=1,2, \dots, r.$ \hfill $\Box$

\section{Detection of directions with wrong line sums}
\label{sec7}

In this section we prove Theorem \ref{mainth}.

\medskip\smallskip
\noindent \textbf{Proof of Theorem 2.1.}\\
We introduce two parameters which may be used to reduce the amount of computation time: we assume that we want to find the correct line sums if in total there are at most
 $F$ wrong line sums and these wrong line sums are in at most $G$ directions. Thus $G \leq F< d/2.$
We prove the following hypothesis by induction on $k$.
\vskip.2cm
\noindent
{\bf Hypothesis for $k$.} {\it The number of directions with wrong line sums that we have already detected equals $r_{k-1}.$ The remaining directions form a set $R_k$ such that if there is a wrong line sum in direction $D_h \in R_k$, then there are at least $k$ wrong line sums in direction $D_h$ and for each direction $D_h \in R_k$,
\begin{equation*}
    \sum_{t \in \mathbb{Z} }t^{j} \ell_{h,t} = \sum_{t \in \mathbb{Z} }t^{j} \ell^*_{h,t}
\end{equation*}
for $j=0$, $1$, $\dots$, $k-2$.}
\vskip.2cm
First we treat the case $k=2$.
We consider the sums $\ell^*_h$ of the line sums $\ell^*_{h,t}$ in each direction $D_h$. Since there are at most $G<d/2$ directions with wrong line sums, the majority of directions has the same correct value $\ell := \sum_{t \in \mathbb{Z}} \ell_{h,t}$ which is the sum of all $f$-values and therefore independent of $h$. We set the, $r_1$ say, directions which have a different sum of line sums apart. We continue with the set $R_2$ of the other $d-r_1$ directions.
Observe that the directions in $R_2$ may have wrong line sums too, but that then in such a direction there are at least two errors, because the sum of the line sums is correct. Thus the Hypothesis holds for $k=2$.

\medskip
Suppose the hypothesis is true for $k$ with $2 \leq k < G$. It follows that the number of directions with wrong line sums in $R_k$ is at most
\begin{equation}
\label{errcount}
    (G-r_{k-1})/k < (d-2r_{k-1})/(2k).
\end{equation}
Hence all directions in $R_k$ have correct line sums if $d-2r_{k-1} \leq 2k$ and if this inequality holds, the induction hypothesis is true for $k+1$. In the sequel we assume
\begin{equation}
\label{dk}
    d-2r_{k-1} \geq 2k+1.
\end{equation}
It follows that $|R_k| = d-r_{k-1} \geq 2k+1$. By renumbering the directions we may assume $D_1, D_2, \dots$, $D_k \in R_k$.
For $h  \in\{1,2, \dots,k\}$, $H>k$  and $K_{k,H}:=\{1,2,\dots, k, H\}$ we define and compute
\begin{equation} \label{cH}
    c^*_{k,H} = \sum_{h=1}^k  E_{h,K_{k,H}} \sum_{t \in \mathbb{Z}} t^{k-1}\ell^*_{h,t} + E_{H,K_{k,H}} \sum_{t \in \mathbb{Z}} t^{k-1}\ell^*_{H,t}.
\end{equation}
From \eqref{casek} with $j=k$ and  \eqref{cH} we obtain, for all $D_H \in R_k, H>k$, similarly to \eqref{diffk},
\begin{equation}
\label{diffk2}
    \sum_{h=1}^k  E_{h,K_{k,H}} \sum_{t \in \mathbb{Z}} t^{k-1}(\ell^*_{h,t}-\ell_{h,t}) + E_{H,K_{k,H}} \sum_{t \in \mathbb{Z}} t^{k-1}(\ell^*_{H,t} -\ell_{H,t})= c^*_{k,H}.
\end{equation}

We distinguish between the following two cases:
\begin{enumerate}[label=\Alph*)]
\item More than $(G-r_{k-1})/k$ directions $D_H \in R_k, H>k$ satisfy $c^*_{k,H} \not = 0$.
\item At most $(G-r_{k-1})/k$ directions $D_H \in R_k, H>k$ satisfy $c^*_{k,H} \not = 0$.
\end{enumerate}

\noindent Case A) Because of the induction hypothesis the number of $H>k$ for which the direction $D_H$ contains a wrong line sum does not exceed $(G-r_{k-1})/k$. Therefore there are at most $(G-r_{k-1})/k$ indices $H>k$ with $D_H \in R_k$ and $E_{H,K_{k,H}} \sum_{t \in \mathbb{Z}} t^{k-1}(\ell^*_{H,t} -\ell_{H,t}) \not= 0$. It follows that there is at least one direction $D_H$ with $ \sum_{h=1}^k  E_{h,K_{k,H}} \sum_{t \in \mathbb{Z}} t^{k-1}(\ell^*_{h,t}-\ell_{h,t}) \not= 0$. This implies that there is an $h \in \{1,2, \dots, k\}$ such that $\sum_{t \in \mathbb{Z}} t^{k-1}(\ell^*_{h,t}-\ell_{h,t}) \not= 0$. Thus there is a direction with a wrong line sum among $D_1, D_2, \dots, D_k$.

\vskip.2cm
\noindent Case B) At least $d-k-r_{k-1}-(G - r_{k-1})/k$ directions $D_H \in R_k$ with $H>k$ have no wrong line sum, hence satisfy $\ell^*_{H,t} = \ell_{H,t}$ for all $t$.
We have, by \eqref{dk},
 $$d-k-r_{k-1} - \frac{d/2 - r_{k-1}}{k} = \frac d2 + \left( \frac d2 -r_{k-1}\right) \left( 1 - \frac 1k \right) - k \geq \frac d2 +\frac 12  -1 - \frac{1}{2k} \geq \frac {d-1}{2} - 1 \geq k-1.$$
Since $k < G < d/2$, at least $k$ directions $d_H \in R_k$ with $H>k$ have no wrong line sums, hence satisfy $\ell^*_{H,t} = \ell_{H,t}$ for all $t$.
Let $D_{k+1}, D_{k+2}, \dots,$ $ D_{2k}$ be directions in $R_k$ with \\$\sum_{t \in \mathbb{Z}} t^{k-1}(\ell^*_{H,t} -\ell_{H,t})= 0$ for all $t$.
Then we have, by \eqref{diffk}, for $H \in \{k+1, k+2, \dots, 2k\}$,
\begin{equation}
\label{diffB}
    \sum_{h=1}^k E_{h,K_{k,H}} \sum_{t \in \mathbb{Z}} t^{k-1}(\ell^*_{h,t}-\ell_{h,t}) = c^*_H=0.
\end{equation}
Here we consider $E_{h,K_{k,H}}$ as coefficients and $\sum_{t \in \mathbb{Z}} t^{k-1}(\ell^*_{h,t}-\ell_{h,t}) $ as unknowns. The coefficient matrix has as typical element
\begin{equation*}
    E_{h,K_{k,H}} =  (-1)^{h-1} \prod_{i,j \in \{1,2, \dots, k,H\},~ i <j,~ i,j \not= h} (a_ib_j - a_jb_i).
\end{equation*}
We claim that the corresponding determinant is nonzero. Observe that the $h$-th column has a nonzero factor
\begin{equation*}
    (-1)^{h-1} \prod_{i,j \in \{1,2, \dots, k\},~ i <j,~ i,j \not= h} (a_ib_j - a_jb_i)
\end{equation*}
in common. By dividing it out for $h=1$, $2$, $\dots$, $k$ the coefficient $E_{h,K_{k,H}}$ reduces to
\begin{equation*}
    E^*_{h,K_{k,H}} := \prod_{i \in \{1,2, \dots, k\},~ i \not= h} (a_ib_H - a_Hb_i).
\end{equation*}
It follows from Lemma \ref{determ} that the determinant of the matrix with typical entry $E^*_{h,D_{k,H}}$ equals
\begin{equation*}
    (-1)^{k(k-1)/2} \left( \prod_{1\leq h_1 <h_2 \leq k} (a_{h_1}b_{h_2} - a_{h_2}b_{h_1}) \right) \left(\prod_{ k+1 \leq H_1  <H_2 \leq 2k} (a_{H_1}b_{H_2} - a_{H_2}b_{H_1}) \right).
\end{equation*}
Since this expression is nonzero, the system \eqref{diffB} has the unique solution $\sum_{t \in \mathbb{Z}} t^{k-1}(\ell^*_{h,t}-\ell_{h,t}) = 0$ for $h = 1,2, \dots, k$.
\vskip.2cm
By comparing the cases A and B we see that $C^*_{k,H} \not= 0$ for at most $(G-r_{k-1})/k$ directions $D_H \in R_k$ with $H>k$ if and only if $\sum_{t \in \mathbb{Z}} t^{k-1}(\ell^*_{h,t}-\ell_{h,t}) = 0$ for $h = 1,2, \dots, k$. Split the $d-r_{k-1}$ directions in $R_k$ into subsets of $k$ elements and a remainder subset of $<k$ elements. Then we have more than $(d-r_{k-1})/k-1$ $k$-subsets. Among them at most $(G-r_{k-1})/k < (d/2 -r_{k-1})/k$ have a direction with a wrong line sum. Since, by \eqref{dk},
$$ \frac{d-r_{k-1}}{k} -1 - \frac{d - 2r_{k-1}}{2k} = \frac{d}{2k}-1  > 0,$$
we see that there is at least one $k$-subset without wrong line sums. Renumber the directions such that this $k$-subset is $\{D_1, D_2, \dots, D_k\}$. Then it follows as in \eqref{dirH} that
$$ E_{H,D_H} \sum_{t \in \mathbb{Z}} t^{k-1}(\ell^*_{H,t} -\ell_{H,t})= c^*_{k,H} $$
for $H>k$ with $D_H \in R_k$.
We define the set $R_{k+1}$ as the set of directions $d_1$, $d_2$, $\dots$, $d_k$ together with the directions $d_H,H>k$ for which $c^*_{k,H}=0$ and define $r_{k}$ as $d - |R_k|$.
For all $d_h \in R_{k+1}$ we have $\sum_{t \in \mathbb{Z}} t^{k-1}(\ell^*_{h,t} -\ell_{h,t})= 0$.
By the induction hypothesis this is also true for the lower powers of $t$. Therefore we have the system of equations
$\sum_{t \in \mathbb{Z}} t^{j}(\ell^*_{h,t}-\ell_{h,t}) = 0$
for $h=1$, $2$, $\dots$, $k$; $j=0$, $1$, $\dots$, $k-1$. It follows from Lemma \ref{vdmo} that $\ell^*_{h,t}-\ell_{h,t} = 0$ for $h=1$, $2$, $\dots$, $k$, if the number of nonzero terms is at most $k$. Thus we may assume that if there is a direction in $R_{k+1}$ with a wrong line sum, then it has at least $k+1$ wrong line sums. This completes the induction step.
\vskip.2cm

We stop detecting directions with wrong line sums at level $k$ if $r_{k-1}=G$ or $k > F -\rho_g$ where $\rho_g$ denotes the number of already detected wrong line sums. If $r_{k-1}=G$, then a wrong line sum in a direction $D_h \in R_{k}$ would lead to a total of $G+1$ directions with wrong line sums which contradicts the definition of $G$. If $k > F -\rho_g$, then a new direction with wrong line sums would give a total of $\rho_g + k > F$ wrong line sums, contradicting the definition of $F$. If a direction in $R_k$ is detected with wrong line sums, then $\rho_g$ is augmented by $k$.
When we stop, we have found all directions with wrong line sums or the assumptions on $F$ and $G$ are not satisfied.
In the latter case one might try a higher value of $F$ or $G$.
\vskip.2cm
It remains to show how the errors can be found and corrected for every direction which contains wrong line sums.
For this we proceed as in the proof of Theorem 2.2. \hfill $\Box$

\section{An error correction algorithm}
\label{sec8}

In this section, we explicitly describe an algorithm for finding directions which contain wrong line sums, and correcting the wrong line sums. Since there may be relatively few errors in practice, we allow the user to specify the maximum number of errors $F$ which have been made in at most $G$ directions, where $G \leq F < d/2$. If $F,G$ are not chosen, set $F=G= \lfloor{(d-1)/2}\rfloor$. We use $\leftrightarrow$ to denote swapped elements. When the line sums of two directions are swapped, $\ell^*_{i,t} \leftrightarrow \ell^*_{j,t}$, it is implicitly meant that this occurs for all $t$.

The algorithm finds the directions that contain wrong line sums, then the wrong line sums themselves, and therafter uses the correct line sums to repair them. We use the variable $g$ to count the number of detected directions containing an erroneous line sum, and order the directions $D=\{D_1, \dots, D_g, D_{g+1}, \dots, D_d\}$, where $D_i$ contains wrong line sums for $i \leq g$. We denote the total number of already detected wrong line sums by $\rho_g$ and the contribution of $D_g$ to it by $\rho^*_g$. Steps 1-7 of the algorithm find all directions for which the sum of line sums does not match the majority and therefore have a wrong line sum. Steps 8-27 detect directions with at least $k \geq 2$ errors which were not detected yet. Steps 28-40 determine the wrong line sums themselves and correct them.

\medskip
In Step 29 we introduce parameter $S=  F - \rho_g + \rho^*_H$ which is an upper bound for the number of wrong line sums in direction $D_H$, since direction $D_H$ already contributed $\rho^*_H$ to $\rho_g$. In Step 31 we use direction $D_{g+2S}$. Since
\begin{equation} \label{relSF}
g + 2S \leq g +2F -2(r_g-r_H) \leq g + 2F - 2(g-1) = 2F -g +2 \leq 2F + 1 \leq d,
\end{equation}
this value of $S$ is permitted. In Step 33 the exact number $s$ of wrong line sums in direction $D_H$ is determined.

\medskip
The computation in Step 35 may not be exact. This is no problem, since the roots $t_1, t_2, \dots, t_s$ are integers and can be found by rounding.

\medskip
To apply Lemma \ref{thmdir} in Steps 36-39, we use Cramer's rule in the form of the matrix determinant lemma. Let $V = (t_j^{i-1})_{i,j = 1}^s$ be the Vandermonde matrix and fix $i$. Let $u^T = (u_j)_{j=1}^r$, $v^T= (v_j)_{j=1}^r$ be column vectors where $u_j = c_j - t_j^{i-1}$ and $v_j$ is equal to 1 for element $i$, and zero elsewhere. Then we can write $x_i$ as
\begin{equation} \label{relVW}
    x_i = \frac{\det(V+uv^T)}{\det(V)} = 1+v^T V^{-1} u = 1+\sum_{j=1}^r V^{-1}_{i,j} u_j.
\end{equation}
Therefore, we do not need to compute determinants for each $i$. Instead, a Vandermonde inverse matrix is computed once.

\medskip
The algorithm may also work well for values of $F$ and $G$ greater than $d/2$. This depends on the way the errors in the line sums are distributed. If, after all, a function $f^*: A \to R$
has been computed, then an easy check reveals whether the line sums of $f^*$ agree with the measured line sums.

\begin{breakablealgorithm}
  \caption{Line sum error correction}\label{euclid}
  \begin{algorithmic}[1]
    \Require{A finite set of (primitive) directions $D = \{(a_h,b_h) : h=1,2, \dots, d\}$ and (measured) line sums $\ell^*_{h,t}$ in the directions of $D$ of a function $f : A \to \mathbb{R}$ such that $\ell^*_{h,t}$ contains at most $F$ errors in at most $G$ directions where $G \leq F < d/2$ ($F$, $G$ may optionally be specified).}
    \vskip.1cm
    \Ensure{Corrected line sums $\ell_{h,t}$.}
    \vskip.3cm
    \For{$h \gets 1$ {\bf to} $d$} \Comment{Find directions with a wrong line sum}
        \State $\ell^*_h \gets \sum\limits_{t \in \mathbb{Z}} \ell^*_{h,t}$
    \EndFor

    \State $g \gets 0$, $\rho_g \gets 0$
    \For{$h \gets 1$ {\bf to} $d$}
        \If{$\ell^*_h \neq$ median$(\{\ell^*_1, \ell^*_2, \dots, \ell^*_d\})$}
            \State $g \gets g+1;~ \rho_g \gets \rho_g+1, \rho^*_g \gets 1$
            \State $D_g \leftrightarrow D_h;~ \ell^*_{g,t} \leftrightarrow \ell^*_{h,t}$
        \EndIf
    \EndFor

    \State $k \gets 2$ \Comment{Find other directions with $k \geq 2$ wrong line sums}
    \While {$k \leq F -\rho_{g} \textbf{ and } g \leq G$}
        \State $maxDirections \gets \min\left(G-g,\left \lfloor{(F -\rho_g)/k}\right \rfloor \right)$	
        \State $i \gets g-k+1$
        \Repeat
            \State $i \gets i+k$
            \State $count \gets 0$
            \For{$H \gets g+1,\dots,i-1,i+k,\dots,d$}
                \State $K_0 \gets \{i,\dots ,i+k-1,H\}$
                \State $c_H \gets \sum\limits_{h=1}^k  E_{g+h,K_0} \sum\limits_{t \in \mathbb{Z}} t^{k-1}\ell^*_{h+i-1,t} + E_{H,K_0} \sum\limits_{t \in \mathbb{Z}} t^{k-1}\ell^*_{H,t}$ \Comment{cf.\eqref{case*}}
                \If{$c_H \neq 0$}
                    \State $count \gets count+1$
                    \If{$count > maxDirections$}
                        \State break
                    \EndIf
                \EndIf
            \EndFor
        \Until{$count \leq maxDirections$}

        \For{$H \gets g+1,\dots,i-1,i+k,\dots,d$}
            \If{$c_H \neq 0$}
                \State $g \gets g+1;~ \rho_g \gets \rho_{g-1}+k; ~\rho^*_g \gets k$
                \State $D_g \leftrightarrow D_H;~ \ell^*_{g,t} \leftrightarrow \ell^*_{H,t}$

            \EndIf
        \EndFor
        \State $k \gets k+1$
    \EndWhile

    \For{$H \gets 1$ {\bf to} $g$}  \Comment{Correct errors in direction $D_H$}
        \State $S \gets F -  \rho_g +\rho^*_H$  \Comment{An upper bound for the number of wrong line sums in $D_H$}

        \For{$j \gets 1$ {\bf to} $2S$}
            \State $K_0 \gets \{g+1,\dots,g+j,H\}$
            \State $c_j \gets \sum\limits_{t \in \mathbb{Z}} t^{j-1} \ell^*_{H,t} + \sum\limits_{h=1}^{j} \dfrac{E_{g+h,K_0}}{E_{H,K_0}} \sum\limits_{t \in \mathbb{Z}} t^{j-1} \ell^*_{g+h,t}$ \Comment{cf. \eqref{case*}}
        \EndFor  \vspace*{3mm}

        \State $s \gets $ rank $ \begin{bmatrix}[lllll]
                        c_{1} & c_{2} & c_{3} & \cdots & c_{S} \\
                        c_{2} & c_{3} & c_{4} & \cdots & c_{S+1} \\
                        c_{3} & c_{4} & c_{5} & \cdots & c_{S+2} \\
                        \multicolumn{1}{c}{\vdots} & \multicolumn{1}{c}{\vdots} & \multicolumn{1}{c}{\vdots} & \ddots & \multicolumn{1}{c}{\vdots} \\
                        c_{S} & c_{S+1} & c_{S+2} & \cdots & c_{2S-1}
                        \end{bmatrix}$

        \State $\begin{bmatrix}			
                B_1\\B_{2}\\\vdots\\B_s
                \end{bmatrix} \gets
                \begin{bmatrix}[llll]
                c_s & -c_{s-1} & \cdots & (-1)^{s-1}c_1 \\
                c_{s+1} & -c_s & \cdots & (-1)^{s-1}c_2 \\
                \multicolumn{1}{c}{\vdots} & \multicolumn{1}{c}{\vdots} & \ddots & \multicolumn{1}{c}{\vdots} \\
                c_{2s-1} & -c_{2s-2} & \cdots & (-1)^{s-1}c_s \\
                \end{bmatrix}^{-1}
                \begin{bmatrix}[l]
                c_{s+1} \\ c_{s+2} \\ \multicolumn{1}{c}{\vdots} \\ c_{2s} \\
                \end{bmatrix}$  \Comment{cf. \eqref{relcB}}\vspace*{3mm}

        \State $t_1,\dots, t_s \gets$ roots($z^s-B_1z^{s-1}+B_2 z^{s-2} - \dots +(-1)^s B_s$) \Comment{cf. \eqref{defB}}

        \State $V \gets (t_j^{i-1})_{i,j=1}^s$
        \State $W \gets V^{-1}$ \Comment{cf. \eqref{relVW}}

        \For{$i \gets 1$ {\bf to} $s$}
            \State $\ell_{H,t_i} \gets \ell^*_{H,t_i} - \sum\limits_{j=1}^s W_{i,j} c_{j}$
        \EndFor

        \State $\rho_g \gets \rho_g + s -\rho^*_H$  \Comment{Total of detected errors}
    \EndFor
\end{algorithmic}
\end{breakablealgorithm}


\section{An example}
\label{sec9}
To illustrate the algorithm we give an example. To show the various aspects of the algorithm, we choose $d=16, F=7, G=4$. Let the directions be given by $D = \{D_1, D_2, \dots, D_{16}\}$ with $D_1=(1,0)$, $D_2=(0,1)$, $D_3=(1,1)$, $D_4=(1,-1)$, $D_5=(2,1)$, $D_6=(2,-1)$, $D_7=(1,2)$, $D_8=(1,-2)$ and eight other directions, e.g. $D_9=(3,1)$, $D_{10}=(3,-1)$, $D_{11}=(1,3)$, $D_{12}=(1,-3)$, $D_{13}=(3,2)$, $D_{14}=(3,-2)$, $D_{15}=(2,3)$, $D_{16}=(2,-3)$. For reason of transparency we assume that all the correct line sums are $0$.
We suppose that there are seven wrong line sums in three directions, $\ell^*_{3,0} = -3$, $\ell^*_{3,4}=3$, $\ell^*_{6,-6}=-2$, $\ell^*_{6,1} = 1$, $\ell^*_{8,3} = 2$, $\ell^*_{8,5} = -4$, $\ell^*_{8,7} =2$. Thus we have two wrong line sums in direction $(1,1)$, two in direction $(2,-1)$ and three in direction $(1,-2)$. We indicate the effects of the steps of the algorithm. Comments are given within square brackets.
\\
\\
\textbf{Steps 4-7.} [Selection of directions with deviant sum of line sums. Exchange of $D_1$ and $D_6$.]\\
median$(\{\ell^*_1, \ell^*_2, \dots, \ell^*_d\})$ = 0 = $\ell^*_h$ for all $\ell^*_h$'s except for $\ell^*_6=-1$, $g \gets 1$, $\rho_1 \gets 1$, $\rho^*_1 \gets 1$, $D_1 \gets (2,-1)$, $D_6 \gets (1,0)$, $\ell^*_{1,-6} \gets -2$, $\ell^*_{1,1} \gets 1$, $\ell^*_{6,-6} \gets \ell^*_{6,1} \gets 0$.
\\
\\
\textbf{Steps 8-21.} [$k=2$, first try. Case A. This will fail since the test directions, $D_2$ and $D_3$, are assumed to have correct line sums, but $D_3$ has wrong line sums. See \eqref{defE} for $E$.] \\$k \gets 2$, {\it maxDirections} $\gets 3$, $i \gets 0$, $i \gets 2$, {\it count} $\gets 0$, $H \gets 4$, $K_0 \gets \{2,3,4\}$, $c_4 \gets  12$,\\ {\it count} $ \gets  1$, $H \gets 5$, $K_0 \gets \{2,3,5\}$, $c_5 \gets  24$, {\it count} $\gets  2$, $H \gets 6$, $K_0 \gets \{2,3,6\}$, $c_6 \gets 12$, \\{\it count} $\gets 3,$ $H \gets 7$, $K_0 \gets \{2,3,7\}$, $c_7 \gets 12$, {\it count} $\gets 4$, break.
\\
\\
\textbf{Steps 12-27.} [$k=2$, second try, Case B. This succeeds since the test directions, $D_4$ and $D_5$, have correct line sums. The double error in $D_3$ will be detected. Exchange of $D_2$ and $D_3$. \\The triple error in $D_8$ will not be detected, since $\sum_t \ell^*_{8,t} = \sum_t t\ell^*_{8,t} = 0$.] \\$i \gets 4$, {\it count} $\gets 0$, $H \gets 2$, $K_0 \gets \{4,5,2\}$, $c_2 \gets 0$, $H \gets 3$, $K_0 \gets \{4,5,3\}$, $c_3 \gets 36$, {\it count} $\gets 1$, $c_H \gets 0$ for $H=6,7, ..., 16$, $H \gets 3$,
 $g \gets 2$, $\rho_2 \gets 3$, $\rho^*_2 \gets 2$, $D_2 \gets (1,1)$, $D_3 \gets (0,1)$, $\ell^*_{2,0} \gets -3$, $\ell^*_{2,4} \gets 3$, $\ell^*_{3,0} \gets \ell^*_{3,4} \gets 0$, $k \gets 3$.
 \\
\\
\textbf{Steps 9-27.} [$k=3$. Since the test directions $D_3, D_4, D_5$ have only correct line sums, this is Case B and the triple error in $D_8$ will be detected. Exchange of $D_8$ and $D_3$.]\\
{\it maxDirections} $\gets 1$, $i \gets 0$, $i \gets 3$, {\it count} $\gets 0$, $K_0 \gets \{3,4,5,H\}$, $c_6 \gets c_7 \gets 0$, $c_8 \gets  -96$, {\it count} $\gets 1$, $c_i \gets 0$ for $i=9, 10, \dots, 16$, $g \gets 3$, $\rho_3 \gets 6$, $\rho^*_3 \gets 3$, $D_3 \gets (1,-2)$, $D_8 \gets (0,1)$, $\ell^*_{3,3} \gets 2$, $\ell^*_{3,5} \gets -4$, $\ell^*_{3,7} \gets 2$, $\ell^*_{8,3} \gets \ell^*_{8,5} \gets
\ell^*_{8,7} \gets 0$, $k \gets 4$.
\vskip.2cm
\noindent [Condition 9 on $F$ is no longer satisfied. The directions with wrong line sums have been detected: $(2,-1), (1,1), (1,-2)$, now $D_1, D_2, D_3$.]
\\
\\
\textbf{Steps 28-40:} [Correction of the line sums for direction  $(2,-1)$. Note that since $\ell^*_{g+h,t} =0$ for $g+h>3$, as in \eqref{dirH}, $c_j$ reduces to $c_j= \sum_{t \in \mathbb{Z}}t^{j-1}\ell^*_{H,t}$.] \\ $H \gets 1$, $S \gets 2$, $j \gets 1$, $K_0 \gets \{4,1\}$, $c_1 \gets -1$, $j \gets 2$, $K_0 \gets \{4,5,1\}$, $c_2 \gets 11$, \\$j \gets 3$, $K_0 \gets \{4,5,6,1\}$, $c_3 \gets -71$, $K_0 \gets \{4,5,6,7,1\}$, $c_4 \gets 431$, \\$s \gets 2$, $B_1 \gets -7$, $B_2 \gets 6$, $t_1 \gets -6$, $t_2 \gets  -1$, $V \gets (1,1 ; -6, -1)$, $W \gets \frac 15 (-1, -1 ; 6, 1)$, $\ell_{1,-6} \gets \ell_{1,1} \gets 0$, $\rho_g \gets 7$.
\\
\\
\textbf{Steps 28-40.} [Correction of the line sums for direction $(1,1)$.] \\ $H \gets 2$,  $S \gets 2$, $c_1 \gets 0$, $c_2 \gets 12$, $c_3 \gets 48$, $c_4 \gets 192$,
$s \gets 2$, $B_1 \gets 4, B_2 \gets 0$, $t_1 \gets 0$, $t_2 \gets 4$, $V \gets (1,1; 0,4)$, $W \gets \frac 14(4, -1; 0, 1)$, $\ell_{2,0} \gets 0$, $\ell_{2,4} \gets 0$, $\rho_g \gets 7$.
\\
\\
\textbf{Steps 28-40.} [Correction of the line sums for direction $(1,-2)$.] \\ $H \gets 3$, $S \gets 3$, $c_1 \gets 0$, $c_2 \gets 0$, $c_3 \gets 16$, $c_4 \gets -240$, $c_5 \gets 2464; c_6 \gets -21600,$ $s \gets 3$, $B_1 \gets -15, B_2 \gets 71, B_3 \gets -105$, $t_1 \gets -7, t_2 \gets -5, t_3 \gets -3$, $V \gets (1,1,1; -7, -5, -3; 49, 25, 9)$, \\$W \gets \frac 18 (15, 8, 1; -42, -20, -2; 35, 12, 1)$, $\ell_{3,-7} \gets \ell_{3,-5} \gets \ell_{3,-3} \gets 0$, $\rho_g \gets 7$.

\vskip.2cm
\noindent [All the wrong line sums have been detected and corrected. After all, it can be checked whether a correct solution has been found indeed by computing the new line sums. If not, the number of wrong line sums exceeded $F$ or the number of directions with wrong line sums exceeded $G$.]

\section{Complexity}
\label{sec10}

In order to compute the complexity of the above algorithm we make some preliminary observations. If there is an $h$ such that $a_h \geq m$ or $|b_h| \geq n$, then each line sum in direction $(a_h,b_h)$ is the $f^*$-value of exactly one point. Without loss of generality we may then assume that $a_h = m$ or $|b_h| = n$, respectively. Hence, for each $h$ the value of $|t|$ in \eqref{defm} is at most $2mn$ and the number of directions $d$ does not exceed $mn$. We further use that $g \leq G \leq F < d/2, h,H \leq d$ and $k \leq G$.

\medskip
In our complexity computation we count an addition, subtraction, multiplication, division and a comparison of two values as one operation. When computing the complexity we do not take into account the size of the terms. (This can be quite high because of the factors $t^j$. )
An operation may therefore mean a multi-precision operation. We assume that the numbers $t^j$ for $0 \leq j <2d-1, |t| \leq 2mn$ are computed once. This involves $\mathcal{O}(dmn)$ operations.

The numbers $t_1$, $t_2$, $\dots$, $t_s$ which are computed in Step 35 of the algorithm are the numbers $t$ indicating the lines of the wrong line sums in direction $H$.  By checking for each integer $\leq 2mn$ with Horner's method whether it is a zero of the polynomial, Step 35 takes $\mathcal{O}(dmn)$ operations.

\medskip
In our analysis, we follow the steps of the pseudocode given in Section \ref{sec8}.
\begin{itemize}
\item {\sl Steps 1-2} (by the above remark on the number of line sums) require $\mathcal{O}(dmn)$ operations (additions).
\item {\sl Steps 3-7} altogether need $\mathcal{O}(dmn)$ operations. (By the Floyd-Rivest algorithm the median can be calculated in linear time, see \cite{fr}.)
\item {\sl Steps 9-11, 27} require $\mathcal{O}(F)$ operations,
\item {\sl Steps 12-14, 22} mean $\mathcal{O}(d/k)$ operations,
\item {\sl Steps 15-17} involve $\mathcal{O}(dk^2mn)$ operations. (According to \eqref{relat} the computation of the $E$'s takes $\mathcal{O}(k^2)$ operations; $k$ does not exceed $F$.)
\item {\sl Steps 18-21} take $\mathcal{O}(G)$ operations,
\item {\sl Steps 23-26} involve $\mathcal{O}(dmn)$ operations.
\end{itemize}
By the structure of this block, the complexity of Steps 9-27 is
\begin{equation*}
    O_{9-11,27}O_{12-14,22} (O_{15-21}+O_{23-26})=\mathcal{O}(d^2mnF^2)
\end{equation*}
where $O_i$ denotes the number of operations in Steps $i$.

\begin{itemize}
\item {\sl Step 28}, in view of $g\leq G$, implies $\mathcal{O}(G)$ repetitions,
\item {\sl Step 29} needs $\mathcal{O}(G)$ additions,
\item {\sl Step 30} implies $\mathcal{O}(F)$ repetitions,
\item {\sl Steps 31-32}, since $j\leq 2F$, require $\mathcal{O}(mnF^3)$ operations,
\item {\sl Step 33} needs $\mathcal{O}(F^2)$ operations,
\item {\sl Step 34}, by Algorithm 2.3.2 on p. 58 of \cite{cohen}, altogether takes $\mathcal{O}(F^3)$ operations,
\item {\sl Step 35}, by Algorithm 2.2.2 on p. 50 of \cite{cohen}, needs $\mathcal{O}(F^3)$ operations,
\item {\sl Step 36}, by an earlier remark, needs $\mathcal{O}(dmn)$ operations,
\item {\sl Steps 37-39} take $\mathcal{O}(F^2)$ operations,
\item {\sl Step 40}, by Algorithm 2.2.2 on p. 50 of \cite{cohen}, requires $\mathcal{O}(F^3)$ operations,
\item {\sl Steps 41-43} need $\mathcal{O}(F^2)$ operations.
\end{itemize}
By the structure of this block, the complexity of Steps 28-43 is given by
\begin{equation*}
    O_{28}(O_{29-32}+O_{33-36}+O_{37-40}+O_{41-43})=\mathcal{O}(mnG(F^3 + d)) = \mathcal{O}(dmnF^2G)
\end{equation*}
where $O_i$ denotes the number of operations implied by the corresponding Steps. Thus the algorithm can  be  completed in $\mathcal{O}(d^2mnF^2)$ operations.


\section{Concluding remarks}
\label{conclusion}

After Ceko, Pagani and Tijdeman \cite{cst} had developed a fast method to reconstruct a consistent discrete tomography problem, the next logical step was to determine what is the most likely set of consistent line sums in case of inconsistency of line sums. If many line sums are almost correct, we refer to Section 4 of \cite{ht13}. In the present paper we study the case that only a small number of line sums is wrong and show how to rectify the wrong line sums. 
We present an algorithm which performs the task in $\mathcal{O}(d^4mn)$ operations. 
However, the numbers involved in an operation may become quite large.

If the domain of $f$ is finite, but not a rectangular grid, then the algorithm can be applied by choosing $A$ as the smallest rectangular grid with sides parallel to the coordinate axes containing the domain and defining function value $0$ for each point of $A$ which does not belong to the domain of $f$. In this way the domain of $f$ is extended to $A$. Hereafter the given algorithm can be used.

An obvious question is whether the reconstruction method for cases with only few wrong line sums can be extended to dimension three and higher. This seems to be hard for two reasons. Firstly a higher dimensional version of Lemma \ref{relat} is wanted in order to be able to detect and correct the wrong line sums. Secondly a three-dimensional version of the algorithm of Ceko, Pagani and Tijdeman is only known under special conditions, see \cite{cpt}. 
The general case might be quite complicated, because it is much more complex to describe the convex hull of the union of all ghosts in dimension $>2$ than in dimension 2.
Thus reconstruction is much more difficult.

Another question is whether it is possible to correct $d/2$ or more errors in the line sums. In a generic case more wrong line sums will be corrected by the algorithm. The example in Section \ref{sec3} shows that it is not always possible to correct the line sums in $d/2$ or more directions, since the $f$-value of one point is uncertain. Theorem 2.2 shows that if the directions with wrong line sums can be detected, per direction quite a few wrong line sums can be corrected.

\subsection*{Acknowledgments}
We thank the referee for his valuable remarks.

The research of L.H. was supported in part by the E\"otv\"os Lor\'and
Research Network (ELKH), by the NKFIH grants 115479, 128088 and 130909 of the Hungarian National Foundation for Scientific Research and by the projects EFOP-3.6.1-16-2016-00022 and EFOP-3.6.2-16-2017-00015, co-financed by the European Union and the European Social Fund.

\end{document}